\documentclass[a4paper, 10pt]{article}
\usepackage{cite}
\usepackage{url}
\usepackage[utf8]{inputenc}
\usepackage[T1,T2A]{fontenc}
\usepackage[utf8]{inputenc}
\usepackage{graphicx}
\usepackage{amsmath}
\usepackage{amssymb,amsfonts,amscd, amsthm}
\usepackage{enumitem}
\newtheorem{lemma}{Lemma}
\newtheorem{theorem}{Theorem}
\newtheorem{corollary}{Corollary}

\theoremstyle{definition}
\newtheorem{definition}{Definition}
\newtheorem{remark}{Remark}
\newtheorem{example}{Example}

\begin{document}

\markboth{Andrej Novikov}{$L_1$-space for a positive operator affiliated with von Neumann algebra}

\title{$L_1$-space for a positive operator affiliated with von Neumann algebra}

\author{ANDREJ NOVIKOV \footnote{Lobachevskii Institute of Mathematics and Mechanics, Kazan (Volga region) Federal University,
a.hobukob@gmail.com}}

\maketitle
 
\begin{abstract}

In this paper we suggest an approach for constructing an $L_1$-type space for a positive selfadjoint operator affiliated with
von Neumann algebra. For such operator we introduce a seminorm, 
and prove that it is a norm if and only if the operator is injective. 
For this norm we construct an $L_1$-type space as the complition of the space of hermitian ultraweakly continuous 
linear functionals on von Neumann algebra, and represent $L_1$-type space as a space of continuous linear functionals
on the space of special sesquilinear forms. Also, we prove that $L_1$-type space is isometrically isomorphic to
the predual of von Neumann algebra in a natural way. We give a small list of alternate definitions
of the seminorm, and a special definition for the case of semifinite von Neumann algebra, in particular.
We study order properties of $L_1$-type space, and demonstrate the connection between semifinite normal weights
and positive elements of this space. At last, we construct a similar $L$-space for the positive element of C*-algebra,
and study the connection between this $L$-space and the $L_1$-type space in case when this C*-algebra is a von Neumann algebra.

\end{abstract}



\section{Introduction}

In 1953 I.E. Segal proposed the foundations of noncommutative integration theory
with respect to a faithful normal semi-finite trace $\tau$,
based on the notion of measurability of an unbounded operator affiliated with a von Neumann algebra,
and used it to define noncommutative $L_1$, $L_2$ and $L_\infty$ spaces.~\cite{Segal1953} Thereafter, this definition was succesfully extended to the full range of $L_p$,
and the theory of this $L_p$-spaces was studied by various authors. Extension of the theory to the case of a von Neumann algebra
equipped with an arbitrary weight became possible only after development of the Tomita–
Takesaki modular theory. The main results in this field are reviewed in \cite{Kostecki-review}.

One of the approaches for noncommutative integration in von Neumann algebras with respect to 
the normal weight was proposed by A.N. Sherstnev in 1970s (see \cite{She74} and \cite{She78}).
Later it was developed by A.M. Bikchentaev, O.E. Tikhonov, N.V. Trunov, A.A. Zolotarev and others. The main results in this area
are reviewed in the  \cite{She82} and \cite{TruShe85}.

In this paper we propose an approach for the construction of $L_1$ space, which is associated with a positive selfadjoint
operator affiliated with a von Neumann algebra. In some sence the statement of a question in this paper is dual to Sherstnev's approach,
and many of the theorems of this paper have their duals, as it is noted in the text.

\section{Definitions and notation}

Throughout this paper we adhere to the following notation. By $\mathcal{M}$ we denote a von Neumann algebra that acts on 
a Hilbert space $H$ with the scalar product $\langle\cdot,\cdot \rangle$. We denote its selfadjoint part by $\mathcal{M}^{sa}$, 
and the set of all projections in $\mathcal{M}$ by $\mathcal{M}^\mathrm{pr}$. Let $p\in \mathcal{M}^\mathrm{pr}$
and $x\in \mathcal{M}$, then by $x_p$ we denote the restriction of $pxp$ to $pH$ ($x_p:=pxp|_{pH}$), also
we denote the reduction of $\mathcal{M}$ to $pH$ by $\mathcal{M}_p$. By $\mathcal{C}(\mathcal{M})$ we denote
the center of $\mathcal{M}$. By $\mathcal{M}_*$ and $\mathcal{M}_*^h$ we denote the 
predual of $\mathcal{M}$ and its Hermitian part, respectively. If an operator $x$ is affiliated with $\mathcal{M}$ then we write $x\eta\mathcal{M}$.
We denote the domain of an operator $x$ by $D(x)$. We denote the closure of an operator $x$ by $\overline{x}$, adjoint operator is denoted as $x^*$. 
We denote the identity operator, the zero operator and the zero vector by $\mathbf{1}$, $\mathbf{0}$ and $\mathit{0}$, respectively. We use standard 
notation for multiplication of a functional $\varphi \in\mathcal{M}_*$ by an operator $x\in \mathcal{M}$, namely, $x\varphi$, $\varphi x$ 
and $x \varphi x$ denote the linear functionals $y \mapsto \varphi(xy)$, $y\mapsto \varphi(yx)$ and $y\mapsto \varphi(xyx)$, respectively.

For a normed space $X$ we use $X^*$ to denote its continuous dual space and $X^{al}$ to denote its algebraic dual space.
For an ordered normed space $X$ we denote its positive cone by $X^+$.

We also consider partial order for positive selfadjoint operators affiliated with $\mathcal{M}$.
For positive selfadjoint $x,\ y \eta \mathcal{M}$ we write $x\leq y$ if and only if $D(y^\frac{1}{2})\subset D(x^\frac{1}{2})$ and
$\|x^\frac{1}{2} f\|^2\leq \|y^\frac{1}{2} f\|^2$ for all $f\in D(y^\frac{1}{2})$. If for an increasing net $(x_j)_{j \in J}$ of 
operators affiliated with $\mathcal{M}$ there exists $x=\sup\limits_{j\in J} (x_j)$, then we write $x_j\nearrow x$.

For a positive selfadjoint operator $x \eta \mathcal{M}$ we use $x_\lambda$ to denote 
$\lambda\overline{x(\lambda+x)^{-1}}$ with $\lambda\in\mathbb{R}^+\setminus\{0\}$.
From the Spectral theorem it follows that the mapping $\lambda \mapsto x_\lambda\in\mathcal{M}^+$ is monotone operator-valued function
and $\lim\limits_{\lambda\to +\infty} x_\lambda^\frac{1}{2} f= x^\frac{1}{2} f$ for all $f\in D(x^\frac{1}{2})$, therefore
$x_\lambda\nearrow x$. For an unbounded $x$ and $\varphi \in \mathcal{M}_*^+$ we define $\varphi(x)$ 
as $\varphi(x):=\lim\limits_{\lambda\to+\infty} \varphi(x_\lambda)$.

\section{Construction and Representation of $L_1$-spaces}

From now on $a$ stands for a positive selfadjoint operator affiliated with $\mathcal{M}$.

In this section we give the definition of $L_1(a)$ and its dual $L_\infty(a)$. Also we construct natural isomorphism of $L_\infty(a)$ onto 
the space of special kind sesquilinear forms $\mathcal{S}_a(\mathcal{M})$.

We consider $\mathfrak{D}_a^{+}\equiv\{\varphi\in \mathcal{M}_*^+ | \varphi(a)<+\infty\}$,
$\mathfrak{D}_a^{h}\equiv\mathfrak{D}_a^+-\mathfrak{D}_a^+$ and $\mathfrak{D}_a\equiv\mathrm{lin}_\mathbb{C} \mathfrak{D}_a^+$. Note that 
if operator $a$ is bounded, then $\mathfrak{D}_a^+=\mathcal{M}_*^+$, $\mathfrak{D}_a^h=\mathcal{M}_*^h$ and $\mathfrak{D}_a=\mathcal{M}_*$.
We define a seminorm $\|\cdot\|_a$ on $\mathfrak{D}_a^{h}$ as
$$\|\varphi\|_a:=\inf\{\varphi_1(a)+\varphi_2(a)\ |
\ \varphi=\varphi_1-\varphi_2;\ \varphi_1,\varphi_2\in\mathfrak{D}_a^+\}.$$
Also, from \cite[Theorem 2]{SkvTik98} states, that if operator $a$ is bounded, 
then $\|\varphi\|_a=\|a^\frac{1}{2} \varphi a^\frac{1}{2}\|$. If $\|\cdot\|_a$ is a norm, then we call it the $a$-norm. Note
that the $\mathbf{1}$-norm coincides with the restriction of the standard norm in $\mathcal{M}_*$ onto $\mathcal{M}_*^h$.

\begin{lemma}\label{density}
 For an injective bounded operator $a$ the sets 
 $a^\frac{1}{2}\mathcal{M}a^\frac{1}{2}=\{a^\frac{1}{2}x a^\frac{1}{2}| x\in \mathcal{M}\}$
 and $a^\frac{1}{2}\mathcal{M}^{sa} a^\frac{1}{2}=\{a^\frac{1}{2}x a^\frac{1}{2}| x\in \mathcal{M}^{sa}\}$ are $\sigma$-weakly dense
 subsets of $\mathcal{M}$ and $\mathcal{M}^{sa}$, respectively.
\end{lemma}

\begin{proof}
According to the Spectral Theorem the sequence
$q_n\equiv a^\frac{1}{2}(\frac{1}{n}+a^\frac{1}{2})^{-1}$ 
is increasing and converges to $\mathrm{rp}(a^\frac{1}{2})=\mathbf{1}$ in the strong operator topology. 
Let $y\in\mathcal{M}$ and $x_n:=(\frac{1}{n}+a^\frac{1}{2})^{-1}y (\frac{1}{n}+a^\frac{1}{2})^{-1}$, then the sequence
$y_n=q_n y q_n= a^\frac{1}{2} x_n a^\frac{1}{2}$ converges to $y$ in the weak operator topology.
Since $\|y_n\|\leq \|y\|$, it follows that $y_n$ $\sigma$-weakly converges to $y$. Hence,
$a^\frac{1}{2}\mathcal{M}a^\frac{1}{2}$ is dense in $\mathcal{M}$ in the $\sigma$-weak operator topology. 
To finish the proof note that if $y$ is selfadjoint, then $y_n$ is selfadjoint.
\end{proof}

A slight change in the latter proof actually shows that the sets $\{\overline{a^\frac{1}{2}x a^\frac{1}{2}}| x\in \mathcal{M}, \overline{a^\frac{1}{2}x a^\frac{1}{2}}\in \mathcal{M}\}$
and $
\{\overline{a^\frac{1}{2}x a^\frac{1}{2}}| x\in \mathcal{M}^{sa}, \overline{a^\frac{1}{2}x a^\frac{1}{2}}\in \mathcal{M}\}$ are $\sigma$-weakly dense
 subsets of $\mathcal{M}$ and $\mathcal{M}^{sa}$, respectively.

\begin{theorem}\label{faithfullness}
$\|\cdot\|_a$ is a norm on $\mathfrak{D}_a^h$ if and only if operator $a$ is injective.
\end{theorem}

\begin{proof}
If operator $a$ is not injective, then there exists a non-zero $f$, such that $af=\mathit{0}$, hence
$\|\langle \cdot f, f\rangle\|_a=\langle a f, f\rangle=0$ and $\|\cdot\|_a$ is not a norm.

Conversely, if $\|\cdot\|_a$ is a norm and operator $a$ is bounded, then 
$\|\varphi\|_a=\|a^\frac{1}{2} \varphi a^\frac{1}{2}\|$ by \cite[Theorem 2]{SkvTik98},
and from Lemma \ref{density}, it follows that $\|a^\frac{1}{2} \varphi a^\frac{1}{2}\|=0$ only if $\varphi=0$. 
If operator $a$ is unbounded, then for each 
$\lambda\in \mathbb{R}^+$ the positive bounded operator $a_\lambda\in \mathcal{M}$ is injective.  
Evidently, $\varphi_1(a_\lambda)+\varphi_2(a_\lambda)\leq \varphi_1(a)+\varphi_2(a)$ for each $\lambda\in\mathbb{R}^+$ and all
$\varphi_1,\ \varphi_2 \in \mathfrak{D}_a^+$,
therefore the inequality $\|\varphi\|_{a_\lambda}\leq\|\varphi\|_a$ holds for each $\varphi\in \mathfrak{D}_a^h$.
Since operator $a_\lambda$ is injective, it follows that $\|\cdot\|_{a_\lambda}$ is a norm, hence $\|\cdot\|_a$ is a norm.

\end{proof}

For an injective operator $a$ by $L_1^h(a)$ we denote the completion of the normed space ($\mathfrak{D}_a^h$, $\|\cdot\|_a$).
By \cite[Proposition 1]{SkvTik98} the dual of $L_1^h(a)$ is 
$(L_\infty^{sa}(a), \|x\|^a)$, where $L_\infty^{sa}(a)\equiv\{x\in(\mathfrak{D}_a^h)^{al} 
|\exists \lambda\in\mathbb{R}, \medskip -\lambda a \leq x \leq \lambda a\}$ and 
$\|x\|^a\equiv\inf\{\lambda\in \mathbb{R}| -\lambda a \leq x \leq \lambda a\}$.
We identify the elements of $\mathfrak{D}_a^h$ with the corresponding elements in $L_1^h(a)$. Further for an injective operator $a$ 
we always assume that $L_\infty^{sa}(a)$ is equiped with the $a$-norm.

For $x\in\mathcal{M}$ we define the sesquilinear form $\widehat{a^\frac{1}{2}xa^\frac{1}{2}}$ on $D(a^\frac{1}{2})\times D(a^\frac{1}{2})$ by
the equality $\widehat{a^\frac{1}{2}xa^\frac{1}{2}}(f,g):=\langle xa^\frac{1}{2}f,a^\frac{1}{2}g\rangle$. The set of all such sesquilinear 
forms is denoted by $\mathcal{S}_a(\mathcal{M})\equiv\{ \widehat{a^\frac{1}{2}xa^\frac{1}{2}} | x\in\mathcal{M}\}$.
We consider partial order on $\mathcal{S}_a(\mathcal{M}^{sa})$, such that $\widehat{a^\frac{1}{2}xa^\frac{1}{2}}\leq \widehat{a^\frac{1}{2}ya^\frac{1}{2}}$ 
if and only if $\widehat{a^\frac{1}{2}xa^\frac{1}{2}}(f, f)\leq \widehat{a^\frac{1}{2}ya^\frac{1}{2}}(f, f)$
for all $f\in D(a^\frac{1}{2})$. By $\mathcal{S}_a(\mathcal{M}^{sa})$ we denote the seminormed space of sesquilinear forms 
$\{\widehat{a^\frac{1}{2}x a^\frac{1}{2}}| x\in \mathcal{M}^{sa}\}$ equiped  with the seminorm $p_a(\widehat{a^\frac{1}{2}x a^\frac{1}{2}})
:=\inf\{ \lambda \in \mathbb{R}^+| -\lambda \widehat{a^\frac{1}{2}\mathbf{1}a^\frac{1}{2}} \leq \widehat{a^\frac{1}{2}x a^\frac{1}{2}}
 \leq \lambda \widehat{a^\frac{1}{2}\mathbf{1}a^\frac{1}{2}}\}$. 

\begin{theorem}\label{definition_afa}
For $\varphi\in\mathfrak{D}_a$ the equality $a^\frac{1}{2}\varphi a^\frac{1}{2}(x)=\lim\limits_{\lambda\to+\infty}\varphi(a^\frac{1}{2}_\lambda x a^\frac{1}{2}_\lambda)$
with $x\in\mathcal{M}$ defines the normal functional $a^\frac{1}{2}\varphi a^\frac{1}{2}\in\mathcal{M}_*$.
\end{theorem}
\begin{proof}
Let $\varphi\in\mathfrak{D}_a^+$ and $\sum\limits_{i=1}^{+\infty}\langle \cdot f_i, f_i\rangle$ be a representation of $\varphi$.
It is evident that $\{f_i\}_{i=1}^{+\infty}\subset D(a^\frac{1}{2})$,
$\sum\limits_{i=1}^{+\infty}\|a^\frac{1}{2}f_i\|^2<+\infty$ and 
$\lim\limits_{\lambda\to +\infty} a^\frac{1}{2}_\lambda f = a^\frac{1}{2} f$ for all $f\in D(a^\frac{1}{2})$,
which implies $\lim\limits_{\lambda\to +\infty}\sum\limits_{i=1}^n\langle x a^\frac{1}{2}_\lambda f_i, a^\frac{1}{2}_\lambda f_i\rangle=
\sum\limits_{i=1}^n\langle x a^\frac{1}{2} f_i, a^\frac{1}{2} f_i\rangle$.
For the fixed $x\in\mathcal{M}$ the inequality $|\langle x a^\frac{1}{2}_\lambda f_i, a^\frac{1}{2}_\lambda f_i\rangle|
\leq \|x\|\|a^\frac{1}{2} f_i\|^2$ holds, hence for the fixed $x\in\mathcal{M}$
the series $\sum\limits_{i=1}^{+\infty}\langle x a^\frac{1}{2}_\lambda f_i, a^\frac{1}{2}_\lambda f_i\rangle$ converge uniformly with
repect to $\lambda$ on $(0,+\infty)$, therefore
$a^\frac{1}{2}\varphi a^\frac{1}{2}(x)\equiv\lim\limits_{\lambda\to +\infty} 
\varphi(a^\frac{1}{2}_\lambda x a^\frac{1}{2}_\lambda)=\lim\limits_{\lambda\to+\infty}\sum\limits_{i=1}^{+\infty}\langle 
x a^\frac{1}{2}_\lambda f_i, a^\frac{1}{2}_\lambda f_i\rangle
=\sum\limits_{i=1}^{+\infty}\langle x a^\frac{1}{2} f_i, a^\frac{1}{2} f_i\rangle$.
The latter implies that $a^\frac{1}{2}\varphi a^\frac{1}{2}(x)$ is well-defined for any $x\in\mathcal{M}$
and $a^\frac{1}{2}\varphi a^\frac{1}{2}\in\mathcal{M}_*^+$.

For $\varphi\in\mathfrak{D}_a$ there exist $\varphi_1$, $\varphi_2$, $\varphi_3$, $\varphi_4\in\mathfrak{D}_a^+$, such that 
$\varphi=\varphi_1-\varphi_2+i\varphi_3-i\varphi_4$.
We have proved that $a^\frac{1}{2}\varphi_i a^\frac{1}{2}\in \mathcal{M}_*^+, i=\overline{1,4}$,
hence $a^\frac{1}{2}\varphi_1 a^\frac{1}{2}(x)-a^\frac{1}{2}\varphi_2 a^\frac{1}{2}(x)
+ia^\frac{1}{2}\varphi_3 a^\frac{1}{2}(x)-ia^\frac{1}{2}\varphi_4 a^\frac{1}{2}(x)=
\lim\limits_{\lambda\to+\infty}\varphi_1(a^\frac{1}{2}_\lambda xa^\frac{1}{2}_\lambda)-\varphi_2(a^\frac{1}{2}_\lambda xa^\frac{1}{2}_\lambda)
+i\varphi_3(a^\frac{1}{2}_\lambda x a^\frac{1}{2}_\lambda)-i\varphi_4(a^\frac{1}{2}_\lambda x a^\frac{1}{2}_\lambda)=
\lim\limits_{\lambda\to+\infty}\varphi(a^\frac{1}{2}_\lambda xa^\frac{1}{2}_\lambda)\equiv a^\frac{1}{2} \varphi a^\frac{1}{2}(x)$,
therefore $a^\frac{1}{2} \varphi a^\frac{1}{2}(x)$ is well-defined for any $x\in\mathcal{M}$ and $a^\frac{1}{2} \varphi a^\frac{1}{2}\in\mathcal{M}_*$.

\end{proof}

\begin{corollary}
For $\varphi\in\mathfrak{D}_a^+$, such that $\varphi=\sum\limits_{i=1}^{+\infty}\langle \cdot f_i, f_i\rangle$, 
the equalities $a^\frac{1}{2}\varphi a^\frac{1}{2}(x)=\sum\limits_{i=1}^{+\infty}\langle x a^\frac{1}{2} f_i, a^\frac{1}{2} f_i\rangle
=\sum\limits_{i=1}^{+\infty}\widehat{a^\frac{1}{2} x a^\frac{1}{2}}(f_i, f_i)$ hold for all $x\in\mathcal{M}$.
\end{corollary}

\begin{definition}\label{8app}
For $\varphi\in \mathfrak{D}_a$ and a sesquilinear form $\widehat{a^\frac{1}{2}xa^\frac{1}{2}}\in\mathcal{S}_a(\mathcal{M})$ we consider
$\varphi(\widehat{a^\frac{1}{2}x a^\frac{1}{2}})\equiv \widehat{a^\frac{1}{2}x a^\frac{1}{2}}(\varphi):=a^\frac{1}{2}\varphi a^\frac{1}{2}(x)$.
\end{definition}

\begin{lemma}\label{sql->linf}
The latter definition indentifies a sesquilinear form $\widehat{a^\frac{1}{2}x a^\frac{1}{2}}\in\mathcal{S}_a(\mathcal{M})$ with
the continuous linear functional on $(\mathfrak{D}_a^h, \|\cdot\|_a)$
\end{lemma}
\begin{proof}
The linearity of $\widehat{a^\frac{1}{2}x a^\frac{1}{2}}$ is evident.
For $\varphi\in\mathfrak{D}_a^h$ consider arbitary $\varphi_1, \varphi_2\in\mathfrak{D}_a^+$, such that $\varphi=\varphi_1-\varphi_2$,
then $|\widehat{a^\frac{1}{2}x a^\frac{1}{2}}(\varphi)|=|a^\frac{1}{2}\varphi a^\frac{1}{2}(x)|\leq
|a^\frac{1}{2}\varphi_1 a^\frac{1}{2}(x)|+|a^\frac{1}{2}\varphi_2 a^\frac{1}{2}(x)|\leq \|x\|(\varphi_1(a)+\varphi_2(a))$,
which implies $|\widehat{a^\frac{1}{2}x a^\frac{1}{2}}(\varphi)|\leq \|x\|\|\varphi\|_a$ and therefore
$\widehat{a^\frac{1}{2}x a^\frac{1}{2}}$ is a continuous linear functional.

\end{proof}

\begin{theorem}\label{linf<->sql}
For an injective operator $a$ the idenitification of the sesquilinear forms in
$\mathcal{S}_a(\mathcal{M}^{sa})$ with the continuous linear functionals on $(\mathfrak{D}_a^h, \|\cdot\|_a)$ lead to the
identification of $(\mathcal{S}_a(\mathcal{M}^{sa}), p_a)$ with $L_\infty^{sa}(a)$.
\end{theorem}
\begin{proof}
Lemma \ref{sql->linf} states that arbitary $\widehat{a^\frac{1}{2}x a^\frac{1}{2}}\in\mathcal{S}_a(\mathcal{M}^{sa})$
is a continuous linear functional on $(\mathfrak{D}_a^h, \|\cdot\|_a)$. Also, $\widehat{a^\frac{1}{2}x a^\frac{1}{2}}$ is real-valued
on $\mathfrak{D}_a^h$, hence $\widehat{a^\frac{1}{2}x a^\frac{1}{2}}\in L_\infty^{sa}(a)$.

Suppose $x\in L_\infty^h(a)$ and $\|x\|^a=1$.
For any $f\in D(a^\frac{1}{2})$ there exists a positive functional $\psi_f:=\langle \cdot f, f \rangle$
and $|x(\psi_f)|\leq\psi_f(a)=\|a^\frac{1}{2}f\|^2$.
Define a form $\widehat{x}(f)$ on $D(a^\frac{1}{2})$ as $\widehat{x}(f):=x(\psi_f)$.
Now, for $f,g$ in $D(a^\frac{1}{2})$ we define a sesquilinear form
$\widehat{x}(f,g)$ on $D(a^\frac{1}{2})\times D(a^\frac{1}{2})$
by the polarization identity
$\widehat{x}(f,g):=\frac{1}{4} (\widehat{x}(f+g)-\widehat{x}(f-g)+i\widehat{x}(f+ig)-i\widehat{x}(f-ig))$.
Note that $x(\varphi)$ is real-valued for every $\varphi$ in $\mathfrak{D}_a^h$.
A straightforward calculation gives the equality $\widehat{x}(f,g)=\overline{\widehat{x}(g, f)}$ and the linearity
by the first argument.
From the above arguments we conclude that
for any $x$ in $L_\infty^h(a)$ there exists a sesquilinear form $\widehat{x}$ on
$D(a^\frac{1}{2})\times D(a^\frac{1}{2})$,
such that $\widehat{x}(f,g)=\overline{\widehat{x}(g, f)}$.

We consider a sesquilinear form $\widehat{y}$ defined
with the equality $\widehat{y}(f,g)=\widehat{x}(a^{-\frac{1}{2}}f, a^{-\frac{1}{2}}g)$
on $\mathrm{Im}(a^{\frac{1}{2}})\times \mathrm{Im}(a^{\frac{1}{2}})$.
From
$|\widehat{y}(f,f)|=|\widehat{x}(a^{-\frac{1}{2}}f, a^{-\frac{1}{2}}f)|\leq \|a^\frac{1}{2} a^{-\frac{1}{2}}f\|^2=\|f\|^2$,
it follows that $\widehat{y}$ is bounded, hence there exists $y\in \mathcal{M}$,
such that $\widehat{y}(f,g)=\langle y f, g \rangle$.
Also, $\widehat{x}(f,g)=\widehat{y}(a^\frac{1}{2}f, a^\frac{1}{2}g)=\langle ya^\frac{1}{2}f, a^\frac{1}{2}g\rangle$
for any $f$ and $g$ in $D(a^\frac{1}{2})$ and $\widehat{x}=\widehat{a^\frac{1}{2}y a^\frac{1}{2}}$ by definition.
Since $\langle y f, g \rangle = \widehat{y}(f,g) = \widehat{x}(a^{-\frac{1}{2}} f, a^{-\frac{1}{2}} g)=
\overline{\widehat{x}(a^{-\frac{1}{2}} g, a^{-\frac{1}{2}} f)}=
\overline{\widehat{y}(g,f)}=\overline{\langle y g, f \rangle}=\langle f, y g \rangle$
it follows that $y$ is selfadjoint. Therefore, we identify $x\in L_\infty^{sa}(a)$ with 
$\widehat{a^\frac{1}{2}ya^\frac{1}{2}}\in \mathcal{S}_a(\mathcal{M}^{sa})$.

The equality $\|x\|^a=\inf\{\lambda|-\lambda a \leq x \leq \lambda a\}=
\inf\{\lambda|-\lambda \widehat{a^\frac{1}{2}\mathbf{1}a^\frac{1}{2}}\leq
\widehat{a^\frac{1}{2} y a^\frac{1}{2}} \leq \lambda \widehat{a^\frac{1}{2}\mathbf{1}a^\frac{1}{2}}\}
=p_a(\widehat{a^\frac{1}{2} y a^\frac{1}{2}})$ completes the proof.

\end{proof}

\begin{remark} \label{B_a(M^sa)->M^sa}
If an operator $a$ is injective then $\inf\{\lambda|-\lambda a \leq a^\frac{1}{2}x a^\frac{1}{2} \leq \lambda a\}=\|x\|$
and the latter implies that the mapping $x\mapsto \widehat{a^\frac{1}{2}x a^\frac{1}{2}}$
is an isometrical isomorphism of $\mathcal{M}^{sa}$ onto $(\mathcal{S}_a(\mathcal{M}^{sa}), p_a)$.
\end{remark}

\begin{corollary}\label{isomorf_u}
For an injective operator $a$ the mapping $u: x\in\mathcal{M}^{sa}\mapsto \widehat{a^\frac{1}{2} x a^\frac{1}{2}}\in L_\infty^{sa}(a)$ 
is an isometrical isomporphism of $\mathcal{M}^{sa}$ onto $L_\infty^{sa}(a)$. Moreover, the adjoint
mapping $u^t$ is an isometrical isomorphism of $(L_\infty^{sa}(a))^*$ onto $(\mathcal{M}^*)^h$.
\end{corollary}

\begin{remark}
For an injective operator $a$ we already identify the elements of $\mathfrak{D}_a^h$ with the corresponding elements of the complition $L_1^h(a)$.
Also, since $(L_\infty^h(a))^*$ is isometrically isomorphic to the second dual of $L_1^h(a)$,
it is convinient to identify the elements of $L_1^h(a)$ with the corresponding elements of $(L_\infty^h(a))^*$.
Then for $u^t$ from Corollary \ref{isomorf_u} and $\varphi\in\mathfrak{D}_a^h$ the equality 
$u^t(\varphi)=a^\frac{1}{2}\varphi a^\frac{1}{2}$ holds, since $u^t(\varphi)(x)=
\varphi(u(x))=\varphi(\widehat{a^\frac{1}{2} x a^\frac{1}{2}})=a^\frac{1}{2}\varphi a^\frac{1}{2}(x)$ for all $x\in\mathcal{M}^{sa}.$
\end{remark}

\begin{definition}
For an injective operator $a$ and a linear functional $\varphi\in(L_\infty^{sa}(a))^*$
we denote $u^t(\varphi)\in(\mathcal{M}^*)^h$ as $a^\frac{1}{2}\varphi a^\frac{1}{2}$, where $u^t$ is 
the isomorphism from Corollary \ref{isomorf_u}.
\end{definition}

\begin{theorem}
For an injective operator $a$ the mapping $v: \varphi \in L_1^h(a) \mapsto a^\frac{1}{2}\varphi a^\frac{1}{2}\in \mathcal{M}_*^h$
is an isometrical isomorphism of $L_1^h(a)$ onto $\mathcal{M}_*^h$.
\end{theorem}

\begin{proof}
Note that $v$ is the restriction of $u^t$ from Corollary \ref{isomorf_u} to $L_1^h(a)$.
Let $\psi$ be an element of $\mathcal{M}_*^{h}$. The sequence 
$\psi_n=(\frac{1}{n}+a^\frac{1}{2})^{-1}\psi (\frac{1}{n}+a^\frac{1}{2})^{-1}$
lies in $\mathfrak{D}_a^h$. Since $a^\frac{1}{2}(\frac{1}{n}+a^\frac{1}{2})^{-1}\nearrow \mathbf{1}$,
it follows that $v(\psi_n)=a^\frac{1}{2}\psi_n a^\frac{1}{2}$ converges to $\psi$ pointwise (i.e.
$\lim\limits_{n\to +\infty} a^\frac{1}{2}\psi_n a^\frac{1}{2}(x)=\psi(x)$ for all $x\in \mathcal{M}$),
hence $v(\mathfrak{D}_a^h)$ is weakly dense in $\mathcal{M}_*^h$.
By \cite[Theorem 3.12]{Rudin}, $v(\mathfrak{D}_a^h)$ is also dense in $\mathcal{M}_*^h$ in the $\mathbf{1}$-norm topology,
and therefore $L_1^h(a)$ is isometrically isomorphic to $\mathcal{M}_*^h$.

\end{proof}

Summarizing the facts, that if operator $a$ is injective, then $(\mathcal{S}_a(\mathcal{M}^{sa}), p_a)$ is isometrically isomorphic to $\mathcal{M}^{sa}$
and that $\mathcal{M}$ is the natural complexification of $\mathcal{M}^{sa}$, it seems reasonable to 
identify the complexification of $\mathcal{S}_a(\mathcal{M}^{sa})$ with $\mathcal{S}_a(\mathcal{M})$ extending 
$p_a$ onto $\mathcal{S}_a(\mathcal{M})$ by the equality $p_a(\widehat{a^\frac{1}{2}xa^\frac{1}{2}})=\|x\|$.
For an injective operator $a$ further we identify $L_\infty(a)$ with $\mathcal{S}_a(\mathcal{M})$,
also we denote complexification of $L_1^h(a)$ as $L_1(a)$ extending $\|\cdot\|_a$ onto $L_1(a)$ by the equality
$\|\varphi\|_a=\|a^\frac{1}{2}\varphi a^\frac{1}{2}\|$.

The folowing corollary evidently follows from Theorem \ref{linf<->sql}
and Corollary \ref{isomorf_u}.

\begin{corollary}\label{L_inf->M}
For an injective operator $a$ the mapping $U: x\in \mathcal{M} \mapsto \widehat{a^\frac{1}{2}xa^\frac{1}{2}}\in L_\infty(a)$
is an isometrical isomorphism of $\mathcal{M}$ onto $L_\infty(a)$,
the adjoint mapping $U^t: \varphi \in L_\infty^*(a) \mapsto a^\frac{1}{2}\varphi a^\frac{1}{2}\in \mathcal{M}^*$,
also is an isometrical isomorphism of $L_\infty^*(a)$ onto $\mathcal{M}^*$. Moreover,
the restriction $V:=U^t|_{L_1(a)}$ is an isometrical isomorphism of $L_1(a)$ onto $\mathcal{M}_*$.
\end{corollary}

The latter Corollary is an analog of the result of \cite{Tru78}.

\section{Alternate Definition of the Norm}

We have defined $\|\varphi\|_a$ as $\inf\{\varphi_1(a)+\varphi_2(a)|
\varphi=\varphi_1-\varphi_2,\varphi_1,\varphi_2\in\mathfrak{D}_a^+\}$. Also, if operator $a$ is bounded, then 
the equality $\|\varphi\|_a=\|a^\frac{1}{2}\varphi a^\frac{1}{2}\|$ holds by \cite[Theorem 2]{SkvTik98}. Thus,
if an operator $a$ is bounded, then we are able to change the definintion of $\|\cdot\|_a$ to that $\|\cdot\|_a$ is the mapping
$\varphi \in \mathfrak{D}_a \mapsto \|a^\frac{1}{2}\varphi a^\frac{1}{2}\| \in \mathbb{R}$ maintaining all the properties described
in the previous section. Also, it is evident, that if operator $a$ is bounded and $\mathcal{M}$ is commutative,
then the equality $\|a^\frac{1}{2}\varphi a^\frac{1}{2}\|=|\varphi|(a)$ holds for all $\varphi\in\mathfrak{D}_a$. Thus,
in case of commutative $\mathcal{M}$ we are able to define $\|\cdot\|_a$ even simplier, 
as the mapping $\varphi\in\mathfrak{D}_a \mapsto |\varphi|(a)\in \mathbb{R}$, maintaining all the properties of $\|\cdot\|_a$.

In this section we demonstrate that in general case the equality $\|\varphi\|_a=\|a^\frac{1}{2}\varphi a^\frac{1}{2}\|$ also holds
for all $\varphi\in \mathfrak{D}_a^h$, therefore $\|\cdot\|_a$ may be alternatively defined as the mapping
$\varphi \in \mathfrak{D}_a \to \|a^\frac{1}{2}\varphi a^\frac{1}{2}\| \in \mathbb{R}$. Also, we show that the
mapping $\varphi\in \mathfrak{D}_a^h \to |\varphi|(a)\in \mathbb{R}$ is a seminorm if and only if 
operator $a$ is affiliated with the center $\mathcal{C}(\mathcal{M})$ of $\mathcal{M}$. Moreover, we show that 
$a\eta \mathcal{C}(\mathcal{M})$ if and only if $\|a^\frac{1}{2}\varphi a^\frac{1}{2}\|=|\varphi|(a)$ for all $\varphi\in \mathfrak{D}_a$.
At last, we consider case of a semifinite $\mathcal{M}$ and $\mathcal{M}=B(H)$ as an example.

Let $a=\int\limits_0^{+\infty} \lambda d e(\lambda)$ be the spectral decomposition of $a$. The equality
$m_a(\varphi)=\int\limits_0^{+\infty} \lambda d \varphi(e(\lambda))$ ($\varphi\in \mathcal{M}_*^+$) defines the additive positively homogeneous
mapping $m_a: \mathcal{M}_*^+ \mapsto [0, +\infty]$, which is lower semicontinuous on $\mathcal{M}_*^+$. The mapping $m_a$ is the element
of extended positive part $\widehat{\mathcal{M}}^+$ of $\mathcal{M}$.\cite{Haa75} For $p\in \mathcal{M}^\mathrm{pr}$ the equality $(pm_a p)(\varphi)=
m_a(p\varphi p)$ ($\varphi \in \mathcal{M}^+_*)$ defines the element $pm_ap$ in $\widehat{\mathcal{M}}^+$, where $(p\varphi p)(x)=
\varphi(pxp) (x\in \mathcal{M}))$.\cite{Haa75}

\begin{lemma}\label{center_affilation_lemma}
Let $a=\int\limits_0^{+\infty} \lambda d e(\lambda)$ be the spectral decomposition of an operator $a$ 
and for any natural $n$ the restriction $a_{e(n)}$ of $a e(n)$ to $e(n) H$ belongs 
to $\mathcal{C}(\mathcal{M}_{e(n)})$, then $a \eta \mathcal{C}(\mathcal{M})$.
\end{lemma}

\begin{proof}
 For a natural $m\geq n$ the restiction $a_{e(n)}$ of $a e(n)$ to $e(m)H$ belongs to $\mathcal{C}(\mathcal{M}_{e(m)})$ as the function
 of the operator $a_m$. Therefore, for any $x\in \mathcal{M}$ the equality
 $a e(n) e(m) x e(m)= e(m) x e(m) a e(n) $ holds. Passing to the limit $m \to +\infty$ we conclude that $a e(n)$ belongs 
 to $\mathcal{C}(\mathcal{M})$, hence $a \eta \mathcal{C}(\mathcal{M})$.
 
\end{proof}

\begin{theorem}\label{center_affilation_theorem}
Let $a=\int\limits_0^{+\infty} \lambda d e(\lambda)$ be the spectral decomposition of an operator $a$,
then the following statements are equivalent:

(i) $a \eta \mathcal{C}(\mathcal{M})$;

(ii) $\forall p\in \mathcal{M}^\mathrm{pr}$ the inequality $pm_a p \leq m_a$ holds;

(iii) $\forall \varphi \in \mathcal{M}_*^h$ and any decomposition $\varphi=\varphi_1-\varphi_2$ ($\varphi_1,\ \varphi_2\in \mathcal{M}_*^+$)
the inequality $m_a(\varphi^+)\leq m_a(\varphi_1)$ holds;

(iv) the mapping $\varphi \mapsto m_a(\varphi^+)$ ($\varphi\in \mathcal{M}_*^h$) is monotone,

i.e., $\varphi, \psi\in \mathcal{M}_*^h$, $\varphi\leq \psi$ imply $m_a(\varphi^+)\leq m_a(\psi^+)$;

(v) the mapping $\varphi \mapsto m_a(\varphi^+)$ ($\varphi\in \mathcal{M}_*^h$) is subadditive,

i.e., $m_a((\varphi+\psi)^+)\leq m_a(\varphi^+)+m_a(\psi^+)$ for all $\varphi, \psi\in \mathcal{M}_*^h$;

(vi) the mapping $\varphi \mapsto m_a(\varphi^+)$ ($\varphi\in \mathcal{M}_*^h$) is convex, i.e.,

$m_a((\lambda\varphi+(1-\lambda)\psi)^+)\leq \lambda m_a(\varphi^+)+(1-\lambda) m_a(\psi^+)$ for all $\varphi, \psi \in \mathcal{M}_*^h$,
$\lambda\in [0,1]$;

(vii) the mapping $\varphi \mapsto m_a(|\varphi|)$ ($\varphi\in \mathcal{M}_*^h$) is subadditive,

i.e., $m_a(|\varphi+\psi|)\leq m_a(|\varphi|)+m_a(|\psi|)$ for all $\varphi, \psi\in \mathcal{M}_*^h$

(viii) the mapping $\varphi \mapsto m_a(|\varphi|)$ ($\varphi\in \mathcal{M}_*^h$) is convex, i.e.,

$m_a(|\lambda\varphi+(1-\lambda)\psi|)\leq \lambda m_a(|\varphi|)+(1-\lambda) m_a(|\psi|)$ for all $\varphi, \psi \in \mathcal{M}_*^h$,
$\lambda\in [0,1]$.
\end{theorem}
\begin{proof}
  First prove the implication $(i)\Rightarrow (ii)$. Since $a=\int\limits_0^{+\infty} \lambda d e(\lambda)$ is affiliated with
  $\mathcal{M}$, it follows that for any natural $n$ the bounded operator $a e(n)$ belongs to $\mathcal{C}(\mathcal{M})$,
  hence $p a e(n) p \leq  a e(n)$ for all $p\in \mathcal{M}^\mathrm{pr}$ by \cite[Theorem 1]{NovTik15}.
  For an arbitary $\varphi \in \mathcal{M}_*^+$ and $p\in \mathcal{M}^\mathrm{pr}$ we have:
  $$ m_a(\varphi) =\lim\limits_{n\to \infty} \varphi(a e(n))\geq \lim\limits_{n\to \infty} \varphi(p a e(n) p)=$$
  $$=\lim\limits_{n\to \infty} (p\varphi p)(a e(n))=m_a(p\varphi p)=(pm_a p)(\varphi)$$
  (i.e. $p m_a p \leq m_a$).
  
  The proof of the implications $(ii) \Rightarrow (iii) \Rightarrow (iv) \Rightarrow (v) \Rightarrow (vii)$, $(v) \Leftrightarrow (vi)$,
 $(vii) \Leftrightarrow (viii)$ literally repeats the proof of respective implications in \cite{NovTik15}, therefore it is sufficient
 to prove the impication $(vii) \Rightarrow (i)$. According to Lemma \ref{center_affilation_lemma}, it is
 sufficient to prove, that for any natural $n$ the operator $a_{e(n)}$ belongs to 
 $\mathcal{C}(\mathcal{M}_{e(n)})$.
 
 Let $(vii)$ hold and $\widetilde{\varphi}, \widetilde{\psi} \in (\mathcal{M}_{e(n)})_*^h$. Construct $\varphi, \psi \in \mathcal{M}_*^h$,
 such that $\varphi=e(n)\varphi e(n)$, $\varphi_{e(n)}=\widetilde{\varphi}$, $\psi=e(n)\psi e(n)$, $\psi_{e(n)}=\widetilde{\psi}$. Then 
 $$ |\widetilde{\varphi}+\widetilde{\psi}|(a_{e(n)})=|\varphi+\psi|(a e(n))=$$ $$=m_a(|\varphi+\psi|)\leq m_a(|\varphi|)+m_a(|\psi|)=
 |\widetilde{\varphi}|(a_{e(n)})+|\widetilde{\psi}|(a_{e(n)}).$$
 By \cite[Theorem 1]{NovTik15}, $a_{e(n)}\in \mathcal{C}(\mathcal{M}_{e(n)})$.
 
\end{proof}

 \begin{remark}
According to \cite[Corollary 1]{NovTik15} if $a\in \mathcal{C}(\mathcal{M})$, then 
$\|a^\frac{1}{2}\varphi a^\frac{1}{2}\|=|\varphi|(a)$ for all $\varphi\in \mathcal{M}_*$. 
We generalize the latter equality to the case of unbounded operator $a$ in Corollary \ref{afa=|f|(a)}.
 \end{remark}
 
In Theorem \ref{definition_afa}  we have defined $a^\frac{1}{2}\varphi a^\frac{1}{2}$ for $\varphi\in\mathfrak{D}_a^h$ as a
normal functional on $\mathcal{M}$ by the equality $a^\frac{1}{2}\varphi a^\frac{1}{2}\equiv
\lim\limits_{\lambda \to +\infty} \varphi(a^\frac{1}{2}_\lambda x a^\frac{1}{2}_\lambda)$ with $a_\lambda \nearrow a$.
Further we define $L_\infty(a)$ for an arbitary $a$ and define the meaning of $a^\frac{1}{2}\varphi a^\frac{1}{2}$ for 
$\varphi\in L_\infty^*(a)$.

\begin{definition}
For a non-injective operator $a$ we define $L_\infty^{sa}(a)$ as the quotient space $\mathcal{S}_a(\mathcal{M}^{sa})/
\{\widehat{x}\in \mathcal{S}_a(\mathcal{M}^{sa})| p_a(\widehat{x})=0\}$.
With $[\widehat{x}]$ we denote the equivalence class of the sesquilinear form $\widehat{x}$. The norm of this
quotient space is denoted by $\|[\widehat{x}]\|_a:=p_a(\widehat{x})$.
\end{definition}

Let $p$ be the projection onto $\ker a$, then for $q=\mathbf{1}-p$ the equalities $aq=qa=a$ hold.

\begin{theorem}\label{la=lap}
The mapping $[\widehat{a^\frac{1}{2}xa^\frac{1}{2}}] \mapsto \widehat{a_q^\frac{1}{2}x_q a_q^\frac{1}{2}}$ is an
isometrical isomorphism of $L_\infty^{sa}(a)$ onto  $\mathcal{S}_{a_q}(\mathcal{M}^{sa}_q)$
\end{theorem}
\begin{proof}
For all $f, g\in D(a^\frac{1}{2})$ the chain of equalities 
$\widehat{a^\frac{1}{2}x a^\frac{1}{2}}(f,g)\equiv\langle x a^\frac{1}{2} f, a^\frac{1}{2} g \rangle=
\langle x qa^\frac{1}{2}q f, qa^\frac{1}{2}q g \rangle=\widehat{a^\frac{1}{2}_q x_q a^\frac{1}{2}_q}(qf,qg)$ holds.
Also, note that $\widehat{a^\frac{1}{2}xa^\frac{1}{2}}(f,g)=\widehat{a^\frac{1}{2}xa^\frac{1}{2}}(qf,qg)$ for all $f,g \in D(a^\frac{1}{2})$.

If $p_a(\widehat{a^\frac{1}{2}xa^\frac{1}{2}}-\widehat{a^\frac{1}{2}ya^\frac{1}{2}})=0$, then 
$-\lambda \widehat{a^\frac{1}{2}\mathbf{1}a^\frac{1}{2}}\leq \widehat{a^\frac{1}{2}xa^\frac{1}{2}}-\widehat{a^\frac{1}{2}ya^\frac{1}{2}}\leq \lambda \widehat{a^\frac{1}{2}\mathbf{1}a^\frac{1}{2}}$
for all $\lambda\geq 0$, which implies $-\lambda \|a^\frac{1}{2} f\|^2 \leq \langle (x-y)a^\frac{1}{2} f, a^\frac{1}{2} f\rangle \leq \lambda \|a^\frac{1}{2}f\|^2$
for all $\lambda\geq 0$ and all $f\in D(a^\frac{1}{2})$. The latter implies that $x_q=y_q$, since
$\mathrm{Im}(a^\frac{1}{2})$ is dense in $qH$. Therefore, the mapping 
$[\widehat{a^\frac{1}{2}xa^\frac{1}{2}}] \mapsto \widehat{a_q^\frac{1}{2}x_q a_q^\frac{1}{2}}$ is well-defined.

The mapping is linear, since $[\widehat{a^\frac{1}{2}xa^\frac{1}{2}}]+[\widehat{a^\frac{1}{2}ya^\frac{1}{2}}]=
[\widehat{a^\frac{1}{2}(x+y)a^\frac{1}{2}}]\mapsto \widehat{a_q^\frac{1}{2}(x+y)_q a_q^\frac{1}{2}}=
\widehat{a_q^\frac{1}{2}(x_q+y_q) a_q^\frac{1}{2}}= \widehat{a_q^\frac{1}{2}x_q a_q^\frac{1}{2}}+ \widehat{a_q^\frac{1}{2}y_q a_q^\frac{1}{2}}$
and $\lambda[\widehat{a^\frac{1}{2}xa^\frac{1}{2}}]=[\widehat{a^\frac{1}{2}(\lambda x)a^\frac{1}{2}}]
\mapsto \widehat{a^\frac{1}{2}_q(\lambda x)_q a^\frac{1}{2}_q}=\widehat{a^\frac{1}{2}_q(\lambda x_q) a^\frac{1}{2}_q}=
\lambda \widehat{a^\frac{1}{2}_qx_q a^\frac{1}{2}_q}$.

Also, the chain of equalities $\|[\widehat{a^\frac{1}{2} x a^\frac{1}{2}}]\|_a=
p_a(\widehat{a^\frac{1}{2} x a^\frac{1}{2}})=
\inf \{\lambda |\forall f\in D(a^\frac{1}{2}) -\lambda \|a^\frac{1}{2} f\|^2 \leq \langle x a^\frac{1}{2} f, a^\frac{1}{2} f\rangle
\leq \lambda \|a^\frac{1}{2} f\|^2\}=
\inf\{\lambda | \forall f \in qH -\lambda\|f\|^2 \leq \langle x_q f, f\rangle \leq \lambda \|f\|^2\}
=\|x_q\|=p_{a_q}(\widehat{a^\frac{1}{2}_q x_q a^\frac{1}{2}_q})$ holds, therefore 
the mapping  $[\widehat{a^\frac{1}{2}xa^\frac{1}{2}}] \mapsto \widehat{a_q^\frac{1}{2}x_q a_q^\frac{1}{2}}$
is an isometrical isomorphism.

\end{proof}

We denote the complexification of $L_\infty^{sa}(a)$ by $L_\infty(a)$. Since $\|[\widehat{a^\frac{1}{2}xa^\frac{1}{2}}]\|_a=
\|x_q\|$ for all $[\widehat{a^\frac{1}{2}xa^\frac{1}{2}}]\in L_\infty^{sa}(a)$, it is natural to extend 
the norm $\|\cdot\|_a$ onto $L_\infty(a)$ with the equality $\|[\widehat{a^\frac{1}{2}xa^\frac{1}{2}}]\|_a=\|x_q\|$.

\begin{definition}\label{noninj_a_aqpa_def}
For $\varphi\in L_\infty^*(a)$ let $\widehat{\varphi}$ be the corresponding element in $\mathcal{S}_{a_q}^*(\mathcal{M}_q)$.
We define $a^\frac{1}{2} \varphi a^\frac{1}{2}$ as a bounded linear functional on $\mathcal{M}$ with the equality
$a^\frac{1}{2} \varphi a^\frac{1}{2}(x):=\varphi([\widehat{a^\frac{1}{2} x a^\frac{1}{2}}])(\equiv\widehat{\varphi}(\widehat{a^\frac{1}{2}_qx_qa^\frac{1}{2}_q}))$.
Note that $a^\frac{1}{2} \varphi a^\frac{1}{2}(x)=a^\frac{1}{2}_q\widehat{\varphi} a^\frac{1}{2}_q(x_q)$.
\end{definition}

\begin{theorem}\label{most_wide_extention}
For any $\varphi$ in $L_\infty^*(a)$ the equality $\|\varphi\|_a=\|a^\frac{1}{2} \varphi a^\frac{1}{2}\|$ holds. 
\end{theorem}
\begin{proof}
 From Theorems \ref{linf<->sql}, \ref{la=lap} and Remark \ref{B_a(M^sa)->M^sa} $L_\infty(a)$ is isometrically isomorphic to $\mathcal{M}_q$. Therefore, the dual $L_\infty^*(a)$ is isometrically
 isomorphic to $\mathcal{M}_q^*$ and $\|\varphi\|_a=\|\widehat{\varphi}\|_{a_q}=\|U^t(\widehat{\varphi})\|=\|a^\frac{1}{2}_q \widehat{\varphi} a^\frac{1}{2}_q\|
 =\|a^\frac{1}{2}\varphi a^\frac{1}{2}\|$, where $U^t$ is the isomorphism from Corollary \ref{L_inf->M}.
 
\end{proof}

\begin{theorem}\label{unbounded_equation}
For any $\varphi$ in $\mathfrak{D}_a^h$ the equality $\|\varphi\|_a=\|a^\frac{1}{2} \varphi a^\frac{1}{2}\|$ holds. 
\end{theorem}
\begin{proof}
 If an operator $a$ is injective, then since $\mathfrak{D}_a^h$ is isometrically embedded into $L_\infty^*(a)$, it follows that
 the equality $\|\varphi\|_a=\|a^\frac{1}{2}\varphi a^\frac{1}{2}\|$ holds 
 for any $\varphi$ in $\mathfrak{D}_a^h$ by Theorem \ref{most_wide_extention}.
 
 If operator $a$ is not injective, then $\|\varphi\|_a=\inf\{\varphi^1(a)+\varphi^2(a)|\varphi=\varphi^1-\varphi^2, \varphi^1,\varphi^2\in\mathfrak{D}_a^+\}=
 \inf\{\varphi^1(qaq)+\varphi^2(qaq)|\varphi=\varphi^1-\varphi^2, \varphi^1,\varphi^2\in\mathfrak{D}_a^+\}
 =\inf\{\widehat{\varphi^1}(a_q)+\widehat{\varphi^2}(a_q)|\widehat{\varphi}=\widehat{\varphi^1}-\widehat{\varphi^2}, \widehat{\varphi^1}, \widehat{\varphi^2} \in\mathfrak{D}_a^+\}=\|\widehat{\varphi}\|_{a_q}=
 \|a^\frac{1}{2}_q\widehat{\varphi} a^\frac{1}{2}_q\|=\|a^\frac{1}{2}\varphi a^\frac{1}{2}\|$.
 
\end{proof}

\begin{remark}
 According to Theorem \ref{linf<->sql}, for $\varphi\in\mathfrak{D}_a$ the normal functional 
 $a^\frac{1}{2}\varphi a^\frac{1}{2}$ from Definition \ref{noninj_a_aqpa_def} and 
 normal functional $a^\frac{1}{2}\varphi a^\frac{1}{2}$ defined in Theorem \ref{definition_afa} 
 coincide as the functionals on $\mathcal{M}$.
\end{remark}

\begin{corollary}\label{||_a=limit}
 For each $\varphi\in \mathfrak{D}_a^h$ the equality $\|\varphi\|_a=\lim\limits_{\lambda\to+\infty}\|\varphi\|_{a_\lambda}$ holds.
\end{corollary}
\begin{proof}
 Evidently, for the fixed $\varphi\in \mathfrak{D}_a^h$ the mapping $\lambda\in \mathbb{R}^+ \mapsto \|\varphi\|_{a_\lambda}$ is monotone 
 and $\|\varphi\|_{a_\lambda}\leq \|\varphi\|_a$, thus for each $\varphi\in \mathfrak{D}_a^h$ there exists $\lim\limits_{\lambda \to +\infty} \|\varphi\|_{a_\lambda}\leq \|\varphi\|_a$.
 
 According to Theorem \ref{unbounded_equation} the chain of equalities $\|\varphi\|_a=\|a^\frac{1}{2}\varphi a^\frac{1}{2}\|=
 |a^\frac{1}{2}\varphi a^\frac{1}{2}|(1)=a^\frac{1}{2}\varphi a^\frac{1}{2}(u)=|a^\frac{1}{2}\varphi a^\frac{1}{2}(u)|$ holds for
 each $\varphi \in \mathfrak{D}_a^h$, 
 where $u|a^\frac{1}{2}\varphi a^\frac{1}{2}|$ is the polar decomposition of $a^\frac{1}{2} \varphi a^\frac{1}{2}$.
 According to Definition \ref{definition_afa}, $|a^\frac{1}{2}\varphi a^\frac{1}{2}(u)|=
 |\lim\limits_{\lambda \to +\infty} a^\frac{1}{2}_\lambda\varphi a^\frac{1}{2}_\lambda(u)|=
 \lim\limits_{\lambda \to +\infty} |a^\frac{1}{2}_\lambda \varphi a^\frac{1}{2}_\lambda(u)|\leq
 \lim\limits_{\lambda \to +\infty} \|a^\frac{1}{2}_\lambda \varphi a^\frac{1}{2}_\lambda\|$ for each $\varphi\in\mathfrak{D}_a^h$.
 Since $\|a^\frac{1}{2}_\lambda \varphi a^\frac{1}{2}_\lambda\|=\|\varphi\|_{a_\lambda}$, it follows
 $\|\varphi\|_a\leq \lim\limits_{\lambda \to +\infty} \|\varphi\|_{a_\lambda} \leq \|\varphi\|_a$ for each $\varphi\in\mathfrak{D}_a^h$.
 
\end{proof}

\begin{corollary}\label{afa=|f|(a)}
 If $a\eta\mathcal{C}(\mathcal{M})$, then $|\varphi|(a)=\|a^\frac{1}{2} \varphi a^\frac{1}{2}\|$ for each $\varphi \in \mathfrak{D}_a^h$,
\end{corollary}
\begin{proof}
 Note that $|\varphi|(a)\equiv\lim\limits_{\lambda \to +\infty} |\varphi|(a_\lambda)$
 and $a_\lambda \in \mathcal{C}(\mathcal{M})$. According to \cite[Corollary 1]{NovTik15},
 $|\varphi|(a_\lambda)=\|a^\frac{1}{2}_\lambda\varphi a^\frac{1}{2}_\lambda\|$ for each $\varphi\in\mathcal{M}_*$. 
 At last, by Corollary \ref{||_a=limit}
 $\lim\limits_{\lambda \to +\infty} \|a_\lambda^\frac{1}{2}\varphi a_\lambda^\frac{1}{2}\|=
 \lim\limits_{\lambda \to +\infty} \|\varphi\|_{a_\lambda}=\|\varphi\|_a$ for each $\varphi \in \mathfrak{D}_a^h$.
 
\end{proof}

Summarizing all the facts of this section, it is possible to define $\|\cdot\|_a$ on $\mathfrak{D}_a^h$ in the different ways,
which are equivalent. First and basic, as the mapping $\varphi \in \mathfrak{D}_a^h \mapsto \inf\{\varphi_1(a)+\varphi_2(a)|
\varphi=\varphi_1-\varphi_2, \varphi_1,\varphi_2\in\mathfrak{D}_a^+\}$, which actually coincides with the mappings
$\varphi\in\mathfrak{D}_a^h\mapsto \inf\{\lim\limits_{\lambda \to +\infty} 
(\varphi_1(a_\lambda)+\varphi_2(a_\lambda))| \varphi=\varphi_1-\varphi_2, \varphi_1,\varphi_2\in \mathcal{M}^+_*\}$
and $\varphi\in\mathfrak{D}_a^h\mapsto \lim\limits_{\lambda\to+\infty} \inf\{\varphi_1(a_\lambda)+
\varphi_2(a_\lambda)|\varphi=\varphi_1-\varphi_2, \varphi_1,\varphi_2\in\mathcal{M}_*^+\}$.
Another one as the mappings $\varphi\in\mathfrak{D}_a^h \mapsto \sup\limits_{\|x\|=1}\lim\limits_{\lambda\to+\infty} 
|\varphi(a^\frac{1}{2}_\lambda x a^\frac{1}{2}_\lambda)|$, which actually coincides
with the mapping $\varphi\in\mathfrak{D}_a^h \mapsto \lim\limits_{\lambda\to+\infty} \sup\limits_{\|x\|=1} 
|\varphi(a^\frac{1}{2}_\lambda x a^\frac{1}{2}_\lambda)|$. At last if $a$ is affiliated with the center $\mathcal{C}(\mathcal{M})$
of an algebra $\mathcal{M}$, then we are able to define $\|\cdot\|_a$ as the mapping $\varphi\in \mathfrak{D}_a^h\mapsto \lim\limits_{\lambda\to+\infty}|\varphi|(a_\lambda)$.
The latter definition is dual to Segal's definition of $\|\cdot\|_1$ as the mapping $x\in \mathfrak{m}_{\tau} \mapsto \tau(|x|)$.\cite{Takesakii}

There is one more possible equivalent definition of $\|\cdot\|_a$ for the case of a semifinite $\mathcal{M}$ with 
a faithfull semifinite normal trace $\tau$. Let $\widetilde{\tau}$ be 
the extension of $\tau$ onto $\mathfrak{m}_\tau=\mathrm{lin}_\mathbb{C} \mathfrak{m}_\tau^+$, where 
$\mathfrak{m}^+_\tau=\{x\in\mathcal{M}^+|\tau(x)<+\infty\}$. We denote $\mathrm{lin}_\mathbb{R} \mathfrak{m}_\tau^+$ 
as $\mathfrak{m}_\tau^{sa}$. By \cite[Theorem V.2.18]{Takesakii}, $x\widetilde{\tau}$ is a normal functional for 
any $x\in\mathfrak{m}_\tau$ and $\|x\widetilde{\tau}\|=\|x\|_\tau=\tau(|x|)$.

\begin{theorem}\label{semifinite}
For any $\varphi=k\widetilde{\tau}$ in $\mathfrak{D}_a$ ($k\in\mathfrak{m}_\tau$), if $\overline{a^\frac{1}{2}ka^\frac{1}{2}}\in\mathcal{M}$, then the equality 
$\|\varphi\|_a=\tau(|\overline{a^\frac{1}{2} k a^\frac{1}{2}}|)$ holds.
\end{theorem}
\begin{proof}
First prove $\overline{a^\frac{1}{2}ka^\frac{1}{2}}\in\mathfrak{m}_\tau$. Assume $k$ positive ($k\widetilde{\tau}\in\mathfrak{D}_a^+$).
For all $f\in D(a^\frac{1}{2})$ the inequality $\|k^\frac{1}{2} a^\frac{1}{2} f\|^2 \leq \|\overline{a^\frac{1}{2}ka^\frac{1}{2}}\|\|f\|^2$
holds. Since  $D(a^\frac{1}{2})$ is dense in $H$, it follows that $\overline{k^\frac{1}{2} a^\frac{1}{2}}\in\mathcal{M}$
and $\|\overline{k^\frac{1}{2} a^\frac{1}{2}}\|\leq \sqrt{\|\overline{a^\frac{1}{2}ka^\frac{1}{2}}\|}$.

Since $\lim\limits_{\lambda \to +\infty} k^\frac{1}{2} a_\lambda^\frac{1}{2} f=k^\frac{1}{2} a^\frac{1}{2} f$ for all 
$f\in D(a^\frac{1}{2})$ and $D(a^\frac{1}{2})$ is dense in $H$, it follows
that $k^\frac{1}{2}a^\frac{1}{2}_\lambda f$ converges to $\overline{k^\frac{1}{2}a^\frac{1}{2}} f$ 
and $a^\frac{1}{2}_\lambda k^\frac{1}{2} f$ converges to $(k^\frac{1}{2} a^\frac{1}{2})^*f$ for all $f\in H$.
At the same time, $\lim\limits_{\lambda \to +\infty} \langle a^\frac{1}{2}_\lambda k^\frac{1}{2} f, a^\frac{1}{2}_\lambda k^\frac{1}{2} f \rangle
=\langle \overline{k^\frac{1}{2} a k^\frac{1}{2}} f, f \rangle$, therefore 
$\overline{k^\frac{1}{2}ak^\frac{1}{2}}=\overline{k^\frac{1}{2} a^\frac{1}{2}}(k^\frac{1}{2} a^\frac{1}{2})^*$.
Moreover, $\overline{a^\frac{1}{2}ka^\frac{1}{2}}=(k^\frac{1}{2} a^\frac{1}{2})^*\overline{k^\frac{1}{2} a^\frac{1}{2}}$, hence
$\tau(\overline{a^\frac{1}{2} k a^\frac{1}{2}})=\tau((k^\frac{1}{2} a^\frac{1}{2})^*\overline{k^\frac{1}{2} a^\frac{1}{2}})=
\tau(\overline{k^\frac{1}{2} a^\frac{1}{2}}(k^\frac{1}{2} a^\frac{1}{2})^*)=\tau(\overline{k^\frac{1}{2} a k^\frac{1}{2}})=
\lim\limits_{\lambda \to +\infty} \tau(k^\frac{1}{2} a_\lambda k^\frac{1}{2})=
k\tau(a)<+\infty$. Therefore, $\overline{a^\frac{1}{2}ka^\frac{1}{2}}\in \mathfrak{m}_\tau^+$.

For all $f\in H$ the chain of inequalities
$\|a^\frac{1}{2}_\lambda k^\frac{1}{2} f\|^2=\langle k^\frac{1}{2} a_\lambda k^\frac{1}{2} f, f \rangle \leq 
\langle \overline{k^\frac{1}{2} a k^\frac{1}{2}} f, f \rangle=\|(k^\frac{1}{2} a^\frac{1}{2})^*f\|^2\leq \|\overline{a^\frac{1}{2} k a^\frac{1}{2}}\|\|f\|^2$
holds. Hence, $\|k^\frac{1}{2} a^\frac{1}{2}_\lambda\|=\|a^\frac{1}{2}_\lambda k^\frac{1}{2}\|\leq \sqrt{\|\overline{a^\frac{1}{2}ka^\frac{1}{2}}\|}$ and
$a^\frac{1}{2}_\lambda k a^\frac{1}{2}_\lambda$ $\sigma$-weakly converges to $\overline{a^\frac{1}{2} k a^\frac{1}{2}}$.
Let $x\in\mathfrak{m}_\tau$, then $a^\frac{1}{2}k\widetilde{\tau} a^\frac{1}{2}(x)=
\lim\limits_{\lambda \to +\infty} k\widetilde{\tau}(a^\frac{1}{2}_\lambda x a^\frac{1}{2}_\lambda)=
\lim\limits_{\lambda \to +\infty} x\widetilde{\tau}(a^\frac{1}{2}_\lambda k a^\frac{1}{2}_\lambda)=
x\widetilde{\tau}(\overline{a^\frac{1}{2} k a^\frac{1}{2}})=\overline{a^\frac{1}{2} k a^\frac{1}{2}}\widetilde{\tau}(x)$.
$\mathfrak{m}_\tau$ is $\sigma$-weakly dense in $\mathcal{M}$, therefore
$a^\frac{1}{2}k\widetilde{\tau} a^\frac{1}{2}=\overline{a^\frac{1}{2} k a^\frac{1}{2}}\widetilde{\tau}$.

By \cite[Theorem V.2.18]{Takesakii} the equality $\|k\widetilde{\tau}\|=\tau(|k|)$ holds for any $k\in\mathfrak{m}_\tau$.
Using Theorem \ref{unbounded_equation} $\|\varphi\|_a=\|a^\frac{1}{2}k \widetilde{\tau} a^\frac{1}{2}\|=\|\overline{a^\frac{1}{2} k a^\frac{1}{2}}\widetilde{\tau}\|=\tau(|\overline{a^\frac{1}{2} k a^\frac{1}{2}}|)$.

\end{proof}

Let $\mathrm{Tr}$ be the canonical trace in the space $B(H)$ of bounded linear operators in Hilbert space $H$ and $C_1(H)$ denote the space 
of trace class operators in $H$.

\begin{corollary}\label{tr|aka|} For any $\varphi=k\widetilde{\mathrm{Tr}}$ in $\mathfrak{D}_a$ ($k\in C_1(H)$) the equality 
$\|\varphi\|_a=\mathrm{Tr}|\overline{a^\frac{1}{2} k a^\frac{1}{2}}|$ holds.
\end{corollary}
\begin{proof}
Assume $\varphi\in \mathfrak{D}_a^+$ ($k\in C_1^+(H)$). Since $\mathrm{Tr}(\overline{k^\frac{1}{2} a k^\frac{1}{2}})=
\lim\limits_{\lambda\to +\infty} \mathrm{Tr}(k^\frac{1}{2} a_\lambda k^\frac{1}{2})=
\lim\limits_{\lambda\to +\infty} k\widetilde{\mathrm{Tr}}(a_\lambda)=k\widetilde{\mathrm{Tr}}(a)<+\infty$, it follows that
$\overline{k^\frac{1}{2} a k^\frac{1}{2}}\in C_1^+(H)$. Hence, $\overline{a^\frac{1}{2} k^\frac{1}{2}}$ is a Hilbert-Schmidt operator
and $\overline{a^\frac{1}{2} k a^\frac{1}{2}}\in C_1^+(H)$.

If $\varphi\in \mathfrak{D}_a$ ($k\in C_1(H)$), then there exist $k_1,\ k_2,\ k_3,\ k_4 \in C_1^+(H)$, such that
$k\widetilde{\mathrm{Tr}}=k_1\widetilde{\mathrm{Tr}}-k_2\widetilde{\mathrm{Tr}}+ik_3\widetilde{\mathrm{Tr}}-ik_4\widetilde{\mathrm{Tr}}$, and 
$\overline{a^\frac{1}{2}k_1 a^\frac{1}{2}}-\overline{a^\frac{1}{2} k_2 a^\frac{1}{2}}+i\overline{a^\frac{1}{2} k_3 a^\frac{1}{2}}-i\overline{a^\frac{1}{2} k_4 a^\frac{1}{2}}=
\overline{a^\frac{1}{2}ka^\frac{1}{2}}\in C_1(H)$.

\end{proof}

The latter Theorem and Corollary make it posible to define $\|\cdot\|_a$ as the mapping
$k\widetilde{\tau}\in\mathfrak{D}_a^h \mapsto \tau(|a^\frac{1}{2}k a^\frac{1}{2}|)$ for the case of
a semifinite $\mathcal{M}$ with a faithfull normal trace $\tau$. This result is similar to the result
of \cite{LugShe84}.

\section{Embedding of Normal Weights into $L_1^+(a)$ and Generation of $L_1(a)$}

Since our approach is influenced by the theory of noncommutative integration with respect to a weight,
one of the natural questions is how our $L_1$-spaces are related to the weights on $\mathcal{M}$.
In this section we show that all semifinite weights, for which $\varphi(a)< +\infty$, can be embedded
into $L_1^+(a)$. Moreover, every element of $L_1^h(a)$ can be represented as the difference of two elements
corresponding to embeddings of semifinite normal weights.

\begin{definition}
 We write $[\widehat{a^\frac{1}{2}x a^\frac{1}{2}}]\in L_\infty^+(a)$ and call
 $[\widehat{a^\frac{1}{2}x a^\frac{1}{2}}]\in L_\infty(a)$ positive if and only if
 $\widehat{a^\frac{1}{2}x a^\frac{1}{2}}\geq \widehat{a^\frac{1}{2}\mathbf{0} a^\frac{1}{2}}$ (equivalently $x_q\geq \mathbf{0}$).
\end{definition}

\begin{definition}
 We write $\varphi\in  (L_\infty^*(a))^+$ and call $\varphi \in L_\infty^*(a)$ positive
 if and only if $\varphi([\widehat{a^\frac{1}{2}x a^\frac{1}{2}}])\geq 0 $ for all $[\widehat{a^\frac{1}{2}x a^\frac{1}{2}}]\in L_\infty^+(a)$.
\end{definition}

 For an injective operator $a$ we identify elements of $L_1(a)$ with the corresponding elements of $L_\infty(a)$.
 By $L_1^+(a)$ we denote the intersection of $(L_\infty^*(a))^+$ with $L_1(a)$.
 
\begin{lemma}\label{pos <-> pos}
 $\varphi\in  (L_\infty^*(a))^+$ if and only if $a^\frac{1}{2}\varphi a^\frac{1}{2} \in (\mathcal{M}^*)^+$.
\end{lemma}

\begin{proof}
 Note that $[\widehat{a^\frac{1}{2} x a^\frac{1}{2}}]\in L_\infty^+(a)$ if and only if $x_q\geq 0$.
 Indeed, if $\widehat{a^\frac{1}{2} x a^\frac{1}{2}}\geq \widehat{a^\frac{1}{2} \mathbf{0} a^\frac{1}{2}}$,
 then $\langle x a^\frac{1}{2} f, x a^\frac{1}{2} f\rangle \geq 0$ for all $f\in D(a^\frac{1}{2})$;
 which implies $\langle x_q f, f\rangle\geq 0$ for all $f\in qH$, since $\mathrm{Im}(a^\frac{1}{2})$ is dense in $qH$.
 Also, if $x_q\geq 0 $, then $\langle x a^\frac{1}{2} f, a^\frac{1}{2} f\rangle = \langle x_q a^\frac{1}{2}_q qf, a^\frac{1}{2}_q qf\rangle\geq 0$.
 
 Assume $\varphi\in  (L_\infty^*(a))^+$. Note that if $x\in \mathcal{M}^+$ then $x_q\in \mathcal{M}_q^+$, hence $[\widehat{a^\frac{1}{2}x a^\frac{1}{2}}]\in L_\infty^+(a)$.
 Using the equality $a^\frac{1}{2}\varphi a^\frac{1}{2}(x)\equiv \varphi([a^\frac{1}{2} x a^\frac{1}{2}])$
 we deduce, that for all $x\in \mathcal{M}^+$ the inequality $a^\frac{1}{2}\varphi a^\frac{1}{2}(x)\geq 0$ holds.
 
 Assume  $a^\frac{1}{2}\varphi a^\frac{1}{2} \in (\mathcal{M}^*)^+$. If $[\widehat{a^\frac{1}{2} x a^\frac{1}{2}}]\in L_\infty^+(a)$,
 then $x_q\in \mathcal{M}_q^+$ and there exists $x'=x_q\oplus 0_{1-q} \in \mathcal{M}^+$, such that
 $[\widehat{a^\frac{1}{2} x a^\frac{1}{2}}]=[\widehat{a^\frac{1}{2} x' a^\frac{1}{2}}]$. Hence,
 $\varphi([\widehat{a^\frac{1}{2} x a^\frac{1}{2}}])=a^\frac{1}{2}\varphi a^\frac{1}{2}(x')\geq 0$ for all 
 $[\widehat{a^\frac{1}{2} x a^\frac{1}{2}}]\in L_\infty^+(a)$.
 
\end{proof}

\begin{corollary}\label{f(1) analoge}
 Let $\varphi \in L_\infty^*(a)$. $\varphi$ is positive if and only if the equality $\|\varphi\|_a=\varphi(a)$ holds.
\end{corollary}
\begin{proof}
 If $\varphi \in L_\infty^+(a)$, then $a^\frac{1}{2} \varphi a^\frac{1}{2} \in (\mathcal{M}^*)^+$ and 
 according to Theorem \ref{most_wide_extention} $\|\varphi\|_a=\|a^\frac{1}{2} \varphi a^\frac{1}{2}\|$.
 Hence, $\|\varphi\|_a=a^\frac{1}{2} \varphi a^\frac{1}{2}(\mathbf{1})=\varphi(a)$.
 
 Conversely, if $\|\varphi\|_a=\varphi(a)$ then $\|a^\frac{1}{2} \varphi a^\frac{1}{2}\|_a=a^\frac{1}{2} \varphi a^\frac{1}{2}(\mathbf{1})$,
 hence $a^\frac{1}{2} \varphi a^\frac{1}{2}\in (\mathcal{M}^*)^+$ and $\varphi \in (L_\infty^*(a))^+$ by Lemma \ref{pos <-> pos}.
 
\end{proof}

\begin{corollary}
 For an injective operator $a$ the isometrical isomorphisms $U, U^t$ and $V$ from Corollary \ref{L_inf->M} preserve the order.
 Moreover, $U(\mathcal{M}^+)=L_\infty^+(a)$, $U^t((L_\infty^*(a))^+)=\mathcal{M}^{*+}$, $V(L_1^+(a))=\mathcal{M}_*^+$.
\end{corollary}




\begin{theorem}\label{posititvity}
For an injective operator $a$ the set $\mathfrak{D}_a^+$ is the dense subset of $L_1^+(a)$.
\end{theorem}
\begin{proof}
Let $\varphi \in L_1^+(a)$. Evidently, $L_1^+(a)\subset L_1^h(a)$ and there exists the sequence $\{\varphi_n\}\subset \mathfrak{D}_a^h$, such that 
$\|\varphi_n -\varphi\|_a \leq \frac{1}{2^n}$. Since $\widehat{a^\frac{1}{2}\mathbf{1} a^\frac{1}{2}}\in L_\infty(a) (\cong L_1^*(a))$,
it follows $\varphi_n(\widehat{a^\frac{1}{2}\mathbf{1} a^\frac{1}{2}}) \to \varphi(\widehat{a^\frac{1}{2}\mathbf{1} a^\frac{1}{2}})=\varphi(a)$.
Let $\varphi_n^1, \varphi_n^2 \in \mathfrak{D}_a^+$ be such that $\varphi_n=\varphi_n^1-\varphi_n^2$
and $\varphi_n^1(a)+\varphi_n^2(a)\leq\|\varphi_n\|_a + \frac{1}{2^n}$. Since $\|\varphi_n-\varphi\|_a\leq \frac{1}{2^n}$,
it follows that $|\varphi_n^1(a)+\varphi_n^2(a)-\|\varphi\|_a| \leq \frac{1}{2^{n-1}}$. Therefore,
$\varphi_n(\widehat{a^\frac{1}{2}\mathbf{1} a^\frac{1}{2}}) + 2\varphi_n^2(a)=\varphi_n^1(a)+\varphi_n^2(a) \to \|\varphi\|_a=\varphi(a)$.
Evidently, $\varphi_n^1(\widehat{a^\frac{1}{2}\mathbf{1} a^\frac{1}{2}})-\varphi_n^2(\widehat{a^\frac{1}{2}\mathbf{1} a^\frac{1}{2}})=
\varphi_n^1(a)-\varphi_n^2(a) \to \varphi(a)$. Hence, $\|\varphi_n^2\|_a=\varphi_n^2(a) \to 0$, which implies $\|\varphi_n^1 - \varphi\|_a\to 0$.

\end{proof}

Let $\Phi$ be a weight on $\mathcal{M}^+$. It is natural to assume that the embedding $\varphi$ of 
$\Phi$ into $L_1(a)$ must be positive, therefore $\|\varphi\|_a=\Phi(a)<+\infty$ by Corollary \ref{f(1) analoge}.

\begin{theorem}\label{weight_embedding}
For an injective operator $a$ any normal weight $\Phi$, such that $\sup\limits_{\lambda\in(0,+\infty)}\Phi(a_\lambda)\equiv\Phi(a)<+\infty$,
defines element of $L_1^+(a)$.
\end{theorem}
\begin{proof}
The condition $\Phi(a)<+\infty$ along with Lemma \ref{density} implies that $\Phi$ is semifinite.
According to \cite{Haa75} 
$\Phi(x)=\sum\limits_{i\in J} \omega_i(x)=\sup\{\sum\limits_{i\in\sigma} \omega_i(x)| \sigma\subset I, \mathrm{card} \sigma \in \mathbb{N}\},
x\in\mathcal{M}^+, \omega_i\in\mathcal{M}_*^+$ for all $x\in\mathcal{M}^+$, which implies that
$\Phi(x)$ is the tight upper bound of all finite sums indexed with the elements of $\sigma\subset J$.

If operator $a$ is bounded, then it is evident, that $\Phi(a)=\sum\limits_{i\in J} \omega_i(a)$. If $a$ is unbounded, then
$\Phi(a_\lambda)=\sum\limits_{i\in J} \omega_i(a_\lambda)\leq \sum\limits_{i\in J} \omega_i(a)$ and $\Phi(a)\leq \sum\limits_{i\in J} \omega_i(a)$. On the other hand,
for the finite $\sigma\subset J$ the inequality $\sum\limits_{i\in\sigma} \omega_i(a_\lambda)\leq \Phi(a_\lambda)$ holds,
therefore $\sum\limits_{i\in\sigma}\omega_i(a)\leq \Phi(a)$. Hence, $\sum\limits_{i\in J}\omega_i(a)=\sup\limits_{\substack{\sigma \subset J,\\
\mathrm{card} \sigma \in\mathbb{N}}} \sum\limits_{i\in \sigma}\omega_i(a)\leq \Phi(a)$, so $\Phi(a)=\sum\limits_{i\in J} \omega_i(a)$.

Since $\Phi(a)<+\infty$, it follows that $\omega_i\in\mathfrak{D}_a^+$. Also, 
since $\ker a =\{\mathit{0}\}$ and $a\geq \mathbf{0}$, it follows $\omega_i(a)=0$ if and only if $\omega_i=0$. Without loss of
generality, we consider $\Phi(a)=\sum\limits_{i=1}^{+\infty} \omega_i(a)$. 
Consider $\varphi_n:=\sum\limits_{i=1}^n\omega_i\in\mathfrak{D}_a^+$.
The sequence $\varphi_n$ is converging in the topology of $a$-norm. If $\|\varphi - \varphi_n\|_a \to 0$
then $\varphi \in L_\infty^+(a)$ and $\|\varphi\|_a=\lim\limits_n \|\varphi_n\|_a = \Phi(a)$. Moreover, for all
$\widehat{a^\frac{1}{2} x a^\frac{1}{2}}\in L_\infty(a)$ the chain of equalities 
$\sum\limits_{i=1}^{+\infty}\omega_i (\widehat{a^\frac{1}{2} x a^\frac{1}{2}})=\lim\limits_n\varphi_n(\widehat{a^\frac{1}{2} x a^\frac{1}{2}})
=\varphi(\widehat{a^\frac{1}{2} x a^\frac{1}{2}})$ holds. If $\widehat{a^\frac{1}{2} x a^\frac{1}{2}}\in L_\infty^+(a)$ is bounded, 
then there exists bounded operator $x_a\in\mathcal{M}^+$, such that $\widehat{a^\frac{1}{2} x a^\frac{1}{2}}(f,g)=\langle x_a f,g)$ for 
all $f,g\in D(a^\frac{1}{2})$ and $\varphi(\widehat{a^\frac{1}{2} x a^\frac{1}{2}})=\Phi(x_a)$.

\end{proof}

\begin{corollary}\label{C1(H)<B(H)}
Let $\mathcal{M}=B(H)$ and $a\in C_1^+(H)\subset B(H)$ be injective. Then 

(i) for any $x\in B^+(H)$ there exists sequence $(x_n)$ in $C_1^+(H)$,

such that $\|x \widetilde{\mathrm{Tr}}- x_n \widetilde{\mathrm{Tr}}\|_a\to 0$;

(ii) for any $x\in B^{sa}(H)$ there exists sequence $(x_n)$ in $C_1^{sa}(H)$,

such that $\|x \widetilde{\mathrm{Tr}}- x_n \widetilde{\mathrm{Tr}}\|_a\to 0$;

(iii) for any $x\in B(H)$ there exists sequence $(x_n)$ in $C_1(H)$,

such that $\|x \widetilde{\mathrm{Tr}}- x_n \widetilde{\mathrm{Tr}}\|_a\to 0$.
\end{corollary}
\begin{proof}
 Since an operator $a$ is bounded, it follows that $\mathfrak{D}_a^+=\mathcal{M}_*^+$. Note that if $x\in B^+(H)$, then 
 $x\widetilde{\mathrm{Tr}}$ is normal semifinite weight, such that $x\widetilde{\mathrm{Tr}}(a)<+\infty$. 
 Applying Theorem \ref{posititvity} and \ref{weight_embedding} we deduce 
 the statement $(i)$. The statements $(ii)$ and $(iii)$ trivialy follow from $(i)$.
 
\end{proof}

\begin{theorem}\label{regularDecomp}
 Any element $\varphi$ of $L_1^h(a)$ can be represented as the difference of two elements $L_1^+(a)$, which are
 embeddings of normal semifinite weights into $L_1(a)$.
\end{theorem}

\begin{proof}
Let ($\omega_i$) be a sequence in $\mathfrak{D}_a^{h}$, such that $\|\omega_n-\varphi\|_a\to 0$. Passing to the subsequence
if it is necessary, we are able to assume that $\|\omega_1\|_a+\sum\limits_{i=1}^\infty\|\omega_{i+1}-\omega_i\|_a<\infty.$
Denoting $\varphi_1=\omega_1$ and $\varphi_n=\omega_n-\omega_{n-1}$ if $n=\overline{2,+\infty}$, we obtain
$\sum\limits_{i=1}^\infty \|\varphi_i\|_a<\infty$, hence $\|\sum\limits_{i=1}^n \varphi_i-\varphi\|_a\to 0$.
Now choose $\varphi_n^1$, $\varphi_n^2$ such that $\varphi_n=\varphi_n^1-\varphi_n^2, 
\varphi_n^1,\varphi_n^2\in\mathfrak{D}_a^+$ and $\varphi_n^1(a)+\varphi_n^2(a)\leq \|\varphi_n\|_a+\frac{1}{2^n}.$
 The sequences $(\sum\limits_{i=1}^n\varphi_i^1)$, $(\sum\limits_{i=1}^n\varphi_i^2)$ are fundamental
in the topology of $\|\cdot\|_a$ and all elements of these sequences are positive,
therefore there exist $\varphi^1,\ \varphi^2\in L_1^+(a)$, such that $\|\varphi^k-\sum\limits_{i=1}^n\varphi_i^k\|_a \to 0$ ($k=\overline{1,2}$).
For $\widehat{a^\frac{1}{2}xa^\frac{1}{2}}\in L_\infty(a)$ the chain of equalities 
$\varphi(\widehat{a^\frac{1}{2}xa^\frac{1}{2}})=\lim\limits_{n\to\infty}\omega_n(\widehat{a^\frac{1}{2}xa^\frac{1}{2}})=
\lim\limits_{n\to\infty}(\sum\limits_{i=1}^n \varphi_i^1(\widehat{a^\frac{1}{2}xa^\frac{1}{2}})-\sum\limits_{i=1}^n \varphi_i^2(\widehat{a^\frac{1}{2}xa^\frac{1}{2}}))=
\lim\limits_{n\to\infty}\sum\limits_{i=1}^n \varphi_i^1(\widehat{a^\frac{1}{2}xa^\frac{1}{2}})-\lim\limits_{n\to\infty}\sum\limits_{i=1}^n \varphi_i^2(\widehat{a^\frac{1}{2}xa^\frac{1}{2}})=
\varphi^1(\widehat{a^\frac{1}{2}xa^\frac{1}{2}})-\varphi^2(\widehat{a^\frac{1}{2}xa^\frac{1}{2}})$ holds.

Let $\Phi^k(x):=\sum\limits_{n=1}^{+\infty} \varphi_n^1(x)$ and $\Phi^2(x):=\sum\limits_{n=1}^{+\infty} \varphi_n^2(x)$
for $x\in \mathcal{M}^+$, $k=\overline{1,2}$. By \cite{Haa75}, $\Phi^1$, $\Phi^2$ are normal weights. Moreover,
$\Phi^k(a)=\sum\limits_{n=1}^{+\infty} \varphi_n^1(a)<+\infty$ ($k=\overline{1,2}$).
By Lemma \ref{density} $\Phi^1,\ \Phi^2$ are semifinite and by Theorem \ref{weight_embedding}
elements $\varphi^1, \varphi^2\in L_\infty^+(a)$ are embeddings of $\Phi^1, \Phi^2$, respectively.

\end{proof}

The latter result is analogue of \cite[Theorem 1]{Tik82}.

\begin{definition}\cite{TruShe85}
 Let $\Phi$ be a normal semifinite weight on $\mathcal{M}$. We call it regular if for any $\varphi\in \mathcal{M}_*^+$ ($\varphi\neq 0$)
 there exists $\omega\in \mathcal{M}_*^+$  ($\omega\neq 0$), such that $\omega\leq \varphi$ and $\omega \leq \Phi$.
\end{definition}

By \cite[Theorem 4]{TruShe85} a normal semifinite weight $\Phi$ on $\mathcal{M}$ is regular if and only if
each sesquilinear form in $L_1^+(\Phi)$ is closable in the sense of \cite{Kato}. By \cite[Theorem 6]{TruShe85} 
a normal semifinite weight on $B(H)$ is regular if and only if $\Phi=k \widetilde{\mathrm{Tr}}$, where $k$ is a positive selfadjoint 
operator in $H$, such that it has the bounded inverse operator.

\begin{theorem}
Let $\mathrm{dim}H = \infty$. For an injective operator $a$ in $C_1^+(H)$, there exists an element $\psi\in L_1^+(a)$, such that $\psi$
cannot be represented as an embedding of a normal semifinite weight.
\end{theorem}
\begin{proof}
Let $\Phi=a\widetilde{\mathrm{Tr}}$ be a positive normal functional on $B(H)$.
From Corollaries \ref{C1(H)<B(H)}, \ref{tr|aka|} and \cite{LugShe84}, it follows that the mapping
$x\in B(H) \mapsto x\widetilde{\mathrm{Tr}}$ determines an isometrical isomorphism of
$L_1(\Phi)$ described in \cite{TruShe85} onto $L_1^h(a)$

If each element of $L_1^+(a)$ can be represented as an embedding of a positive normal semifinite weight,
then by \cite[Theorem 5.12]{Pedersen-Takesaki} for each element $\psi$ of $L_1^+(a)$ there exists the correspoding
selfadjoint operator $k_\psi\geq0$, such that $\psi=k_\psi\widetilde{\mathrm{Tr}}$. The
corresponding sesquilinear form $\widehat{k_\psi}\in L_1^+(\Phi)$ 
($\widehat{k_\psi}(f,g):=\langle k^\frac{1}{2}_\psi f, k^\frac{1}{2}_\psi g \rangle$) is closable by \cite[Theorem 1.27]{Kato}.

By \cite[Theorem 6]{TruShe85} the weight $\Phi$ is not regular.
Hence, by \cite[Theorem 4]{TruShe85} there exists a positive nonclosable sesquilinear form in $L_1^+(\varphi)$,
so we get a contradiction.

\end{proof}

\section{Case of C*-algebras}

It is notable that the same approach can be applied to the case of C*-algebras.
Let $\mathcal{A}$ be a C*-algebra. By $\mathcal{A}^+$, $\mathcal{A}^{*+}$ we denote its positive cone
and the positive cone of its continuous dual. By $\mathcal{A}^{sa}$ we denote the set of all selfadjoint
operators in $\mathcal{A}$, by $\mathcal{A}^{*h}$ we denote the set of all continuous hermitian functionals
on $\mathcal{A}$. For $a\in \mathcal{A}^+$
we define $\|\cdot\|_a$ on $\mathcal{A}^*$ as the mapping $f \in \mathcal{A}^* \mapsto \|a^\frac{1}{2} f a^\frac{1}{2}\|$.

\begin{theorem}
For all $f \in \mathcal{A}^{*h}$ the equality $\|f\|_a=\inf \{ f_1(a)+f_2(a)| f=f_1-f_2, f_1, f_2 \in \mathcal{A}^{*+}\}$
holds.
\end{theorem}
\begin{proof}
Let $\pi$ be the embedding of $\mathcal{A}$ into the universal enveloping von Neumann algebra $\mathcal{N}$. For $f\in \mathcal{A}^{*h}$
let $\varphi$ be the corresponding element of $\mathcal{N}_*^h$, then $\|f\|_a=\|a^\frac{1}{2} f a^\frac{1}{2}\|=
\|\pi(a)^\frac{1}{2} \varphi \pi (a)^\frac{1}{2}\|$. Hence, according to \cite[Theorem 2]{SkvTik98} the equality 
$\|f\|_a=\|\varphi\|_{\pi(a)}$ holds. Therefore, $$\|f\|_a=\inf\{\varphi_1(\pi(a))+\varphi_2(\pi(a)) | \varphi=\varphi_1-\varphi_2,
\varphi_1, \varphi_2 \in \mathcal{N}_*^{+}\}=$$ $$=\inf\{f_1(a)+f_2(a) | f=f_1-f_2, f_1, f_2\in \mathcal{A}^{*+}\}$$.

\end{proof}

\begin{theorem}\label{C*-norm}
$\|\cdot\|_a$ is a norm on $\mathcal{A}^*$ if and only if $\varphi(a)>0$ for all $\varphi\in\mathcal{A}^{*+}\setminus\{\mathbf{0}\}$. 
\end{theorem}
\begin{proof}
 If $\varphi$ is positive, then $\|a^\frac{1}{2}\varphi a^\frac{1}{2}\|=a^\frac{1}{2}\varphi a^\frac{1}{2}(\mathbf{1})=\varphi(a)$.
Hence, if $\varphi(a)=0$ and $\varphi\neq 0$ then $\|\cdot\|_a$ is not a norm.

 Using the embedding $\pi$ of $\mathcal{A}$ into the universal enveloping von Neumann algebra $\mathcal{N}$ we
get a positive selfadjoint injective operator $\pi(a)\in \mathcal{N}$. Using Theorem \ref{faithfullness} we get that
$\|\cdot\|_{\pi(a)}$ is a norm on $\mathcal{N}_*$. Since $\mathcal{A}^*\cong \mathcal{N}_*$ and from
\cite[Theorem 2]{SkvTik98}, it follows that $\|\cdot\|_a$ is a norm.

\end{proof}

For an operator $a$, such that $\varphi(a)>0$ for all $\varphi\in \mathcal{A}^{*+}\setminus \{0\}$,
we construct $M(a)$ as the complition of $(\mathcal{A}^*,\|\cdot\|_a)$. Also, according to Theorem \ref{L_inf->M}
the mapping $\varphi \in M(a) \mapsto a^\frac{1}{2}\varphi a^\frac{1}{2} \in \mathcal{A}^*$
is an isometrical isomorphism of $M(a)$ onto $\mathcal{A}^*$.

\begin{remark}\label{bounded_inverse}
If an operator $a$ has the bounded inverse operator, then $\|\cdot\|_a$ is a norm on $\mathcal{A}^*$,
since $\frac{1}{\|a^{-1}\|}\|\cdot\|\leq \|\cdot\|_a\leq \|a\|\|\cdot\|$. Also, the
latter inequalities imply, that $L_1(a)$ coincides with $\mathcal{A}^*$ as the topological vector spaces.
\end{remark}

The condition $\varphi(a)>0$ for all $\varphi\in\mathcal{A}^{*+}\setminus\{0\}$ has varios interpretations for
various C*-algebras. If we consider $a=(a_n)\in\mathcal{A}=c_0$, then this condition is equivalent to $\forall n\in \mathbb{N}\ a_n>0$.
If we consider $a=(a_n)\in\mathcal{A}=c$, then this condition implies $\forall n\in \mathbb{N}\ a_n>0$ and $\lim a_n >0$,
hence for operator $a=(a_n)$ there exists the bounded inverse operator $a^{-1}=(\frac{1}{a_n})\in c$.

\begin{remark} Assume $\mathcal{A}$ is a von Neumann algebra.
If an operator $a\in \mathcal{A}$ satisfies the conditions of Theorem \ref{C*-norm}, then it also
satisfies the conditions of Theorem \ref{faithfullness}. Therefore, we are able to construct $L_1(a)$ and $M(a)$.
Evidently, we are able to naturally embed $L_1(a)$ into $M(a)$ as a linear subspace. If $\dim(H)=+\infty$, then
$\mathcal{A}_*$ and $\mathcal{A}^*$ do not coincide. According to Corollary \ref{isomorf_u} $L_1(a)$ and $M(a)$ are isometrically isomorphic
to $\mathcal{A}_*$ and $\mathcal{A}^*$, respectively. Therefore, if $\dim H=+\infty$, then $L_1(a)$ and $M(a)$ do not coincide.
\end{remark}

\begin{example}
To give the example of such operator $a$, that satisfy the condition of Theorem \ref{faithfullness}, but
does not satisfy the condition of Theorem \ref{C*-norm}, consider $\mathcal{A}=\ell_\infty$.

Since $\ell_\infty$ is an abelian von Neumann algebra, which acts on the Hilbert space $H=\ell_2$, 
we are able to construct $L_1(a)$ for an injective operator $a$. The injectiveness of an operator $a$ is
equivalent to the condition, that $a_n>0$ for each $n\in\mathbb{N}$. For example, $(a_n)=(\frac{1}{n})\in \ell_\infty^+$ is injective.
According to Remark \ref{bounded_inverse} if $a$ has a bounded inverse operator, then $a$ satisfies the conditions of Theorem \ref{C*-norm}.
Let us prove that if $a$ satisfies the conditions of Theorem \ref{C*-norm}, then $a$ has the bounded inverse operator.
If we assume the contrary, then eather there exists $a_n=0$, or there exists a subsequence $a_{n_k}$, such that $\lim a_{n_k}=0$.
Evidently, for each $a_{n}$ there exists a functional $\varphi_n\in \ell_\infty^{*+}$, such that $\varphi_n(a):=a_n$,
therefore $\forall n\in \mathbb{N} \ a_n>0$. Also for each $a\in \ell_\infty^+$ there exists a Banach limit 
$\varphi_a\in \ell_\infty^{*+}$, such that $\varphi_a(a)=\liminf\limits_{n\to\infty}a_n$. Hence, $\liminf\limits_{n\to\infty}a_n>0$.
Therefore, $a$ has the bounded inverse operator.

However, $(\frac{1}{n})$ does not have a bounded inverse operator, therefore it does not satisfy the conditions of Theorem \ref{C*-norm}.
\end{example}

\begin{example}

To give a noncommutative example, assume $\mathcal{A}=B(H)$ and $\mathrm{dim} H=\infty$, then there exists an injective positive trace-class operator $a\in C^+_1(H)$.
For any injective operator $a$ we are able to construct $L_1(a)$, but for any trace-class operator
there exists a Dixmier trace $\varphi\in B^{*+}(H)$,
for which $\varphi(a)=0$. Therefore, such operator $a$ does not satisfy the conditions of Theorem \ref{C*-norm}.
\end{example}

\section*{Acknowledgment}

Research supported in part by Russian Foundation for Basic Research grant 14-01-31358.

I would like to thank Dr. Oleg Tikhonov from Kazan Federal University for his expert advice (especially on the Theorem 5) 
and encouragement. He kindly read my paper and offered invaluable detailed advices on grammar, organization, and the theme of the paper.

\end{document}